\ifx\pdfoutput\undefined\else\pdfoutput=1\fi
\documentclass[9pt,a4paper,oneside,reqno]{amsart} 
\usepackage[english]{babel}
\usepackage[p,osf]{cochineal}
\usepackage[scale=.95,type1]{cabin}
\usepackage[zerostyle=c,scaled=.94]{newtxtt}
\usepackage[T1]{fontenc} 
\usepackage[utf8]{inputenc} 
\usepackage{amsthm}

\usepackage{amsmath}
\usepackage{amssymb}
\usepackage{pgfplots} 
\pgfplotsset{compat=newest} 
\usepgfplotslibrary{fillbetween,colorbrewer}
\usepackage{amsfonts}
\usepackage{amsaddr}
\usepackage{xspace,subcaption}
\usepackage{enumitem} 
\usepackage{csquotes}
\usepackage{xcolor}
\usepackage[colorlinks]{hyperref}
\hypersetup{colorlinks, breaklinks, citecolor=teal,filecolor=black}
\usepackage[nosort,capitalise]{cleveref}
\crefname{enumi}{}{}
\crefname{equation}{}{}
\usepackage{graphicx}
\usepackage{mathtools}
\mathtoolsset{centercolon}
\usepackage{bm}
\usepackage{pgfkeys}
\usepackage{pgf,interval}
\intervalconfig{soft open fences}

\DeclareMathOperator{\Tr}{Tr}
\DeclareMathOperator{\erfc}{erfc}

\DeclareMathOperator{\supp}{supp}

\DeclareMathOperator{\E}{\mathbf{E}}
\DeclareMathOperator{\Prob}{\mathbf{P}}

\DeclareMathOperator{\Spec}{Spec}

\DeclareMathOperator{\dett}{det_2}

\newcommand{\ov}{\overline}
\newcommand{\ii}{\mathrm{i}}

\newcommand{\F}{\mathbf{F}}

\ifdefined\C
\renewcommand{\C}{\mathbf{C}}
\else
\newcommand{\C}{\mathbf{C}}
\fi
\newcommand{\HC}{\mathbf{H}}

\newcommand{\vx}{\bm{x}}

\newcommand{\vy}{\bm{y}}

\newcommand{\wt}{\widetilde}

\newcommand{\R}{\mathbf{R}}
\newcommand{\N}{\mathbf{N}}
\newcommand{\Z}{\mathbf{Z}}

\newcommand{\cO}{\mathcal{O}}
\newcommand{\co}{{\scriptstyle\mathcal{O}}}

\newcommand{\dif}{\operatorname{d}\!{}}
\DeclarePairedDelimiter{\abs}{\lvert}{\rvert}%
\DeclarePairedDelimiter{\norm}{\lVert}{\rVert}%
\providecommand\given{}
\newcommand\SetSymbol[1][]{\nonscript\:#1\vert\allowbreak\nonscript\:\mathopen{}}
\DeclarePairedDelimiterX{\tuple}[1](){\renewcommand\given{\SetSymbol[\delimsize]}#1}
\DeclarePairedDelimiterX{\set}[1]\{\}{\renewcommand\given{\SetSymbol[\delimsize]}#1}
\DeclarePairedDelimiterXPP{\landauO}[1]{\cO}(){}{#1}
\DeclarePairedDelimiterXPP{\landauo}[1]{\co}(){}{#1}
\DeclarePairedDelimiterXPP{\landauOprec}[1]{\cO_\prec}(){}{#1}

\usepackage[giveninits=true,url=false,doi=false,isbn=false,eprint=true,datamodel=mrnumber,sorting=nty,maxcitenames=4,maxbibnames=99,backref=false,block=space,backend=biber,bibstyle=phys]{biblatex} 
\AtEveryBibitem{\clearfield{month}}
\AtEveryCitekey{\clearfield{month}} 
\renewbibmacro{in:}{}
\DeclareFieldFormat[article]{title}{\emph{#1}} 
\DeclareFieldFormat{mrnumber}{\ifhyperref{\href{http://www.ams.org/mathscinet-getitem?mr=#1}{\nolinkurl{MR#1}}}{\nolinkurl{#1}}}
\DeclareFieldFormat{pmid}{\ifhyperref{\href{https://www.ncbi.nlm.nih.gov/pubmed/#1}{\nolinkurl{PMID#1}}}{\nolinkurl{#1}}}
\DeclareFieldFormat{eprint}{\ifhyperref{\href{https://arxiv.org/abs/#1}{\nolinkurl{arXiv:#1}}}{\nolinkurl{#1}}}
\renewbibmacro*{doi+eprint+url}{%
  \iftoggle{bbx:doi}{\printfield{doi}}{}
  \newunit\newblock%
  \printfield{mrnumber}%
  \newunit\newblock%
  \printfield{pmid}%
  \newunit\newblock%
  \printfield{eprint}%
  \iftoggle{bbx:url}{\usebibmacro{url+urldate}}{}}
\author{Giorgio Cipolloni}
\address{Princeton Center for Theoretical Science, Princeton University}
\author{L\'aszl\'o Erd\H{o}s$^{\dagger}$ \and Yuanyuan Xu$^{\dagger}$}
\address{IST Austria, Am Campus 1, A-3400 Klosterneuburg, Austria}
\author{Dominik Schr\"oder$^{\ddagger}$}
\address{Institute for Theoretical Studies, ETH Zurich, Clausiusstrasse 47, 8092 Zurich, Switzerland}
\email{dschroeder@ethz.ch}
\email{gc4233@princeton.edu}
\email{lerdos@ist.ac.at} 
\email{yuanyuan.xu@ist.ac.at}
\thanks{$^\dagger$Supported by the ERC Advanced Grant ``RMTBeyond'' No.~101020331}
\thanks{$^\ddagger$Supported by Dr.\ Max R\"ossler, the Walter Haefner Foundation and the ETH Z\"urich Foundation}
\title{Directional Extremal Statistics for Ginibre Eigenvalues}
\date{\today}

\newtheorem{theorem}{Theorem}
\newtheorem{definition}[theorem]{Definition}

\newtheorem{proposition}[theorem]{Proposition}
\newtheorem{lemma}[theorem]{Lemma}

\begin{document}
\thispagestyle{empty}
 
 \newcommand{\cob}{\color{blue}}
 \newcommand{\nc}{\normalcolor}

\begin{abstract} We consider the eigenvalues of a large dimensional real or complex Ginibre matrix  in 
    the region of the complex plane where their real parts reach their maximum value. %
     This maximum  follows
    the Gumbel distribution  and that these extreme eigenvalues  form a Poisson point process,
    asymptotically as the dimension 
    tends to infinity. In the complex case these facts have already been established by Bender~\cite{MR2594353} and
    in the real case by Akemann and Phillips~\cite{MR3192169} even for the more general elliptic ensemble
    with a sophisticated saddle point analysis. The purpose of this note is to give a very short direct proof in the Ginibre
    case  with an effective error term.  
     Moreover, our estimates on the correlation kernel in this regime
    serve as a key input for  accurately locating  $\max\Re\Spec(X)$
    for any large matrix $X$ with i.i.d.\ entries in the companion paper~\cite{2206.04448}.
\end{abstract}
\maketitle

\section{Introduction} 

The Ginibre matrix ensemble~\cite{MR173726} is the simplest and most commonly used prototype of non-Hermitian random matrices.
It consists of $n\times n$ matrices $X$ with
independent, identically distributed (i.i.d.) Gaussian entries $x_{ij}$. We use the normalization $\E x_{ij}=0$, 
$\E \abs{x_{ij}}^2= \frac{1}{n}$, i.e. $\sqrt{n}x_{ij}$ is a standard real or complex normal  random variable.  Correspondingly,
we talk about real or complex Ginibre matrices.
The empirical density of eigenvalues converges to the uniform distribution on the unit disk in the complex
plane, known as \emph{Girko's circular law} and proven in increasing 
generality even without Gaussian assumption~\cite{MR773436,MR1428519,MR2409368},
while the spectral radius converges to 1~\cite{MR866352,MR863545,MR3813992,MR4408512}
with an explicit speed of convergence~\cite{MR4408013}.
For the Gaussian case, the eigenvalues form a determinantal  (or Pfaffian)
point process with an explicit correlation kernel $K_n(z, w)$
(see~\eqref{det} and~\eqref{realkernel} later). This kernel 
was computed by Ginibre in the complex case~\cite{MR173726} and later by Borodin and Sinclair for the more complicated real case~\cite{MR2530159,MR3537345} based upon earlier works on Pfaffian formulas~\cite{17930739, MR2371225} (some special cases have been  solved earlier~\cite{MR1121461, MR1437734, MR1231689,MR2439268,MR2185860}
 and see also~\cite{MR3380679} for a comprehensive summary of all known related kernels). 
While the eigenvalue distribution is rotationally symmetric in the complex case, 
the main complication in the real case stems from the fact that the real axis plays a special role, in 
fact there are many real eigenvalues~\cite{MR1231689}.   

The explicit formula for the eigenvalue correlation function allows one, in principle, to compute the distribution of any
interesting statistics of the eigenvalues. In reality, these calculations may require very precise asymptotic
analysis of certain special functions where  the complex  and real cases may differ substantially.
For example, the distribution of $\rho(X):= \max \abs{\Spec(X)}$,  the spectral radius of $X$ (i.e.\ 
the largest eigenvalue in modulus), can still be easily identified in the complex case by using 
Kostlan's observation~\cite{MR1148410} on the moduli of the complex Ginibre eigenvalues. The precise result,
stated in this form  in~\cite{MR1986426}, asserts that 
\begin{equation}\label{Crho}
    \rho(X) \stackrel{\text{d}}{=} 1 + \sqrt{\frac{\alpha_n}{4n}} + \frac{1}{\sqrt{4n\alpha_n}} G_n, \qquad \alpha_n:=\log n -2\log\log n-
    \log (2\pi),
\end{equation}
where $G_n$ converges in distribution to a standard \emph{Gumbel random variable}, i.e.
\[
\lim_{n\to \infty} \Prob (G_n\le t) = \exp{(-e^{-t})}
\]
for any fixed $t\in \R$. On the other hand, lacking radial symmetry, which is key element of Kostlan's observation, the analogous result for the real Ginibre  ensemble required a much more sophisticated analysis by Rider and Sinclair in~\cite{MR3211006}. They showed
that~\eqref{Crho} also holds for real case with the same
scaling factor $\alpha_n$, but $G_n$ 
converges to a slightly rescaled Gumbel law with distribution function $\exp{(-\frac{1}{2}e^{-t})}$.
The additional factor $1/2$ stems from the fact that the spectrum of a real Ginibre matrix is
symmetric with respect to the real axis. 

In the current paper we investigate a related quantity, the maximum real part of the spectrum of $X$,
where radial symmetry does not help even in the complex case.
It turns out that a similar %
asymptotics holds but with a new scaling  factor: 
\begin{equation}\label{Cmax}
    \max\Re \Spec(X) \stackrel{\text{d}}{=} 
    1 + \sqrt{\frac{\gamma}{4n}} + \frac{1}{\sqrt{4n\gamma}} G_n, \qquad \gamma=\gamma_n:=
    \frac{\log n-5\log\log n-\log (2\pi^4)}{2},
\end{equation}
with $G_n$ still converging to a Gumbel variable. More precisely: %
\begin{theorem}[Gumbel distribution]\label{thm gumbel}
    Let \(\sigma_1,\ldots,\sigma_n\) denote the eigenvalues of a real (\(\beta=1\)) 
    or complex (\(\beta=2\)) \(n\times n\) Ginibre matrix. Then for any fixed~\footnote{
    Our estimates actually  give a slightly weaker effective error for any $|t|\ll \sqrt{\log n}$}  \(t\in\R\) it holds that 
    \begin{equation}\label{gum}
        \begin{split}
            \Prob\Bigl(\max_i\Re\sigma_i< 1 + \sqrt{\frac{\gamma}{4n}}+\frac{t}{\sqrt{4\gamma n}} \Bigr)
            = \exp\Bigl(-\frac{\beta}{2}\exp(-t)\Bigr) + \landauO*{  \frac{ (\log\log n)^2}{\log n} }
        \end{split}
    \end{equation}
    as $n\to\infty$.
\end{theorem}
In the complex case~\eqref{gum} as a  limit statement  was proven by Bender~\cite{MR2594353}
and in the real case by Akemann and Phillips in~\cite{MR3192169} even for the  more involved
elliptic Ginibre ensemble where the kernel $K_n$ is expressed by a contour integral 
(later it was extended to the  chiral two-matrix model with complex entries~\cite{MR2761338}).
Here we give a short alternative proof that also provides an effective estimate on the speed of convergence.

In Theorem~\ref{thm gumbel} we only considered the eigenvalue with the largest real part for simplicity,
however similar result holds for the largest eigenvalue in any chosen direction. More precisely, in the complex
case the distribution of  $\max_i \Re (e^{\ii \theta} \sigma_i)$ is independent of $\theta\in \R$  by rotational symmetry.
For  real Ginibre matrices and  for any fixed $\theta\ne 0 $ independent of $n$,
 $\max_i \Re (e^{\ii \theta} \sigma_i)$ still satisfies~\eqref{gum} but 
with $\beta=2$. Our proof can easily be extended to cover this more general case using 
 that the local eigenvalue correlation functions for real and complex Ginibre matrices
practically coincide away from the real axis.

As a motivation we remark that $\max \Re \Spec(X)$ is the basic quantity determining the exponential growth rate of the  long time asymptotics of 
the solution of the linear system of differential equations
\[
\frac{\dif}{\dif t}{\bf u}(t) = X {\bf u}(t).
\]
Starting from the pioneering work of May~\cite{4559589}
(see also the more recent review~\cite{Allesina2015}), 
this equation is frequently used in phenomenological models to describe the evolution of many
interacting agents  with random couplings
both in theoretical neuroscience~\cite{10039285,17155583} and in 
mathematical ecology~\cite{25768781,26198207}.

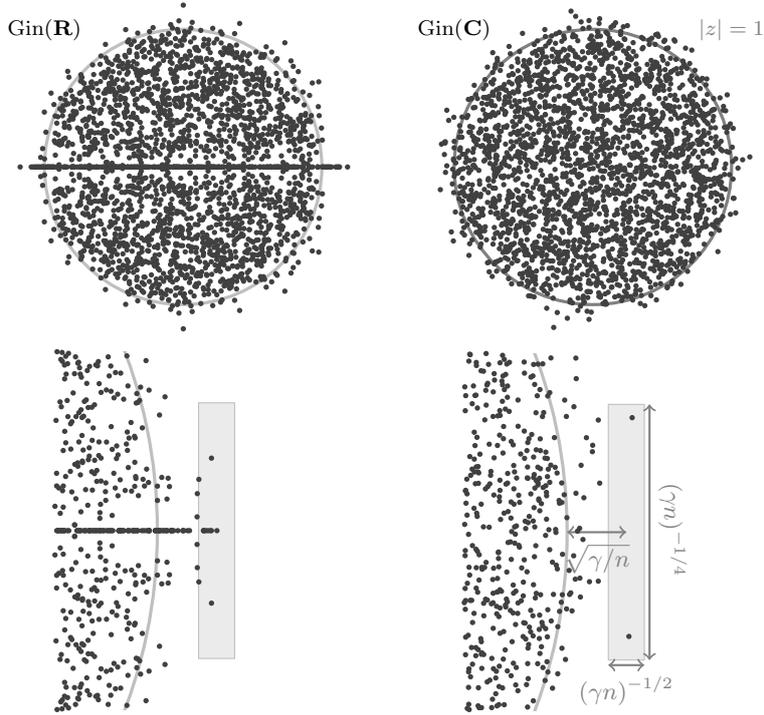
\begin{figure}
    \begin{tikzpicture}
    \begin{axis}[xmin=-1.3,xmax=1.3,ymin=-1.3,ymax=1.3,width=20em,height=20em,axis line style={white},ticks=none]
        \addplot [only marks,draw=none,mark size=1pt,mark options={draw opacity=0,fill=darkgray}] table [col sep=comma] {500r.csv};
        \draw(axis cs:0,0) [draw=lightgray,very thick] circle[radius=1];
        \node  at (axis cs:-1,1) {\small\(\mathrm{Gin}(\R)\)};
    \end{axis}
\end{tikzpicture}\qquad
\begin{tikzpicture}
    \begin{axis}[xmin=-1.3,xmax=1.3,ymin=-1.3,ymax=1.3,width=20em,height=20em,axis line style={white},ticks=none]
        \addplot [only marks,draw=none,mark size=1pt,mark options={draw opacity=0,fill=darkgray}] table [col sep=comma] {500c.csv};
        \draw(axis cs:0,0) [draw=gray,very thick] circle[radius=1];
        \node at (axis cs:1,1) {\small\color{gray}\(|z|=1\)};
        \node  at (axis cs:-1,1) {\small\(\mathrm{Gin}(\C)\)};
    \end{axis}
\end{tikzpicture}\\
\begin{tikzpicture}
    \begin{axis}[xmin=0.7,xmax=1.4,ymin=-0.35,ymax=0.35,width=20em,height=20em,ticks=none,axis line style={white}]
        \addplot [only marks,draw=none,mark size=1pt,mark options={draw opacity=0,fill=darkgray}] table [col sep=comma] {1000r.csv};
        \draw(axis cs:0,0) [draw=lightgray,very thick] circle[radius=1];
        \draw [fill=lightgray,fill opacity=.3,draw=lightgray] (axis cs:1.08,.25) -- (axis cs:1.15,.25) -- (axis cs:1.15,-.25) -- (axis cs:1.08,-.25) -- cycle;
    \end{axis}
\end{tikzpicture}\qquad 
\begin{tikzpicture}
    \begin{axis}[xmin=0.7,xmax=1.4,ymin=-0.35,ymax=0.35,width=20em,height=20em,axis line style={white},ticks=none]
        \addplot [only marks,draw=none,mark size=1pt,mark options={draw opacity=0,fill=darkgray}] table [col sep=comma] {1000c.csv}; 
        \draw [fill=lightgray,fill opacity=.3,draw=lightgray] (axis cs:1.08,.25) -- (axis cs:1.15,.25) -- (axis cs:1.15,-.25) -- (axis cs:1.08,-.25) -- cycle;
        \draw(axis cs:0,0) [draw=lightgray,very thick] circle[radius=1];
        \draw[<->,thick,black!50]
        (axis cs:1.16,.25) -- (axis cs:1.16,-.25) node[midway,above=1pt,rotate=270] {$(\gamma n)^{-1/4}$};
        \draw[<->,thick,black!50]
        (axis cs:1.08,-.26) -- (axis cs:1.15,-.26) node[midway,below=1pt] {$(\gamma n)^{-1/2}$};
        \draw[<->,thick,black!50]
        (axis cs:1.,0) -- (axis cs:1.115,0) node[midway,below=1pt] {$\sqrt{\gamma/n}$};
    \end{axis}
\end{tikzpicture}
\caption{The figure shows the eigenvalues of real and complex Ginibre matrices. The eigenvalues for the top figures have been computed for \(50\) independent Ginibre matrices of size \(50\times 50\), while for the bottom figure \(100\) independent matrices of size \(100\times 100\) have been sampled. Note that the eigenvalues of the real Ginibre matrix are symmetric with respect to the real axis, and that some (in fact \(\sim\sqrt{n}\)) eigenvalues are on the axis itself. Furthermore, the top left figure misleadingly hints that the rightmost eigenvalue is
real. This is a finite $n$ effect (see~\cite{MR3211006} for a detailed
discussion of this so called ``Saturn effect''); we actually prove (see~\crefrange{Gumbel real eq}{Gumbel real real eq}) that in the large $n$ limit the largest real eigenvalue is much smaller than the real part of the rightmost complex eigenvalue.}\label{fig circ law}
\end{figure}

The  appearance of the universal Gumbel distribution in~\eqref{Crho}--\eqref{Cmax} is typical for extreme value statistics of independent random variables as one of the three main cases described in
the Fisher-Tippet-Gnedenko theorem. While nearby Ginibre eigenvalues 
inside the unit disk are strongly correlated, the extreme eigenvalues 
are essentially independent which heuristically explains the Gumbel law. 
The key point is that  the correlation length of the eigenvalues is of order $n^{-1/2}$, 
as the scaling of the  Ginibre kernel $K_n(z, w)$ indicates, but in the extreme regime
the few eigenvalues that may contribute to $\rho(X)$ or $\max\Re \Spec(X)$  are much farther away from each other
than $n^{-1/2}$.
In fact, the scaling factor $\gamma=\gamma_n$ is chosen in such a way that there are typically 
finitely many (independent of $n$) eigenvalues in an elongated box of size 
\((4\gamma n)^{-1/2}\times \ii (\gamma n)^{-1/4}\) around \(1+\sqrt{\gamma/4n}\) (see~\cref{fig circ law}).
The height of this box, which is essentially the square root of its width, is determined by the curvature of
the boundary of the circular law: above or below this box there are no eigenvalues since their modulus
would be too large. Given this heuristic picture, the 
typical distance between the eigenvalues in the relevant box 
is of order $n^{-1/4}$  modulo logarithmic factors, so they are well beyond the
correlation scale hence independent.   As a second result, we also establish this independence rigorously;
in fact we show that within this box the eigenvalues form 
a Poisson point process in the $n\to\infty$ limit.
Again, as a pure limit statement  this result has already been proven in~\cite{MR2594353}
for the complex Ginibre ensemble 
and in~\cite{MR3192169} for the real case; our contribution is to give an alternative direct proof
with an effective error bound. 
\begin{theorem}[Poisson Point Process]\label{thm poisson} 
    Let \(\sigma_1,\ldots,\sigma_n\) denote the eigenvalues of a real or complex \(n\times n\) Ginibre matrix. 
    Fix any \(t\in\R\) and any 
    function \(f\colon\C\to[0,\infty)\) supported on \([t,\infty)\times\ii\R\), 
     which, additionally, is assumed to be symmetric \(f(z)=f(\ov z)\) in the real case\footnote{This restriction is only for convenience,
     since by spectral symmetry $\overline{\sigma(X)}=\sigma(X)$, any non-symmetric function \(f\)
      can be replaced by its symmetrization \([f(z)+f(\ov z)]/2\)}.  Then we have 
    \begin{equation}\label{pois}
        \begin{split}
            &
            \E e^{-\sum_{i=1}^n  f(x_i+\ii y_i)}
            = \exp\Bigl(-\int_\F (1-e^{-f(x+\ii y)}) \frac{e^{-x-y^2}}{\sqrt{\pi}}\dif y\dif x\Bigr)
            + \landauO*{  \frac{ (\log\log n)^2}{\log n} },
        \end{split}
    \end{equation}
    where we introduced the eigenvalue rescaling
    \begin{equation}
        \sigma_i = 1+\sqrt{\frac{\gamma}{4n}} + \frac{x_i}{\sqrt{4\gamma n}} + \frac{\ii y_i}{(\gamma n)^{1/4}}
    \end{equation}
    and we set \(\mathbf F=\HC:=\set{z\in\C\given \Im z\ge 0}\) in the real and \(\mathbf F=\C\) in the complex case.
\end{theorem}

Both our main results follow from a precise asymptotics
of the rescaled Ginibre kernel $K_n(z, w)$ in the relevant box combined with 
the idea of the regularized Fredholm determinant also used in~\cite{MR3211006}.
The compact form of $K_n$ in the Ginibre case makes the calculations considerably shorter 
than the saddle point analysis for its contour integral representation used for the 
elliptic ensemble in~\cite{MR2594353, MR3192169}. In particular, we obtain an effective bound
on the speed of convergence unlike~\cite{MR2594353, MR3192169} that rely on dominated convergence. 
As a byproduct, we also obtain  concentration result with an effective error term for the linear statistics (in particular the number) of 
eigenvalues on a slightly larger box. This result  is crucially used in our companion paper~\cite{2206.04448}
in which we
accurately identify the size of $\max\Re \Spec(X)$ for matrices with general  i.i.d.\ entries,
going  well beyond the explicitly solvable models.

We close this introduction with a remark about eigenvectors. For many Hermitian random matrices or 
operators originating from disordered quantum systems,  the general  prediction is that 
Poisson eigenvalue statistics entails localized eigenvectors (while strongly correlated eigenvalue 
statistics, e.g. Wigner-Dyson, imply delocalized eigenvectors). This is not the case here: all
eigenvectors, even those  corresponding to extreme eigenvalues in the Poisson regime are fully
delocalized~\cite[Corollary 2.4]{MR4408013}.

\subsection*{Acknowledgement} We are grateful to G.\ Akemann for bringing the references \cite{MR2761338, MR3192169, MR2594353,MR3380679}
to our attention. 
Discussions with  Guillaume Dubach on a preliminary version of this project  are gratefully
acknowledged. 

\section{Complex Ginibre}
We recall a few basic facts about the correlation functions.
The joint probability density of the eigenvalues of a complex Ginibre matrix is given by~\cite{MR1986426}
\begin{equation}
    \rho_n(\bm z)=\rho_n(z_1,\ldots,z_n): = \frac{n^n}{\pi^n 1!\cdots n!} \exp\Bigl(-n\sum_{i}\abs{z_i}^2\Bigr) \prod_{i<j}\bigl(n\abs{z_i-z_j}^2\bigr).
\end{equation}
The product can be written as a product of Vandermonde determinants and we obtain 
\begin{equation}\label{det}
    \begin{split}
        \prod_{i<j}\bigl(n\abs{z_i-z_j}^2\bigr) &= \det\begin{pmatrix}
            1 & \sqrt{n}z_1 & \cdots & (\sqrt{n}z_1)^{n-1} \\ 
            \vdots & \vdots & \ddots & \vdots \\
            1 & \sqrt{n}z_n & \cdots & (\sqrt{n}z_n)^{n-1}
        \end{pmatrix}\begin{pmatrix}
            1 & \cdots & 1 \\ 
            \sqrt{n}\ov{z_1} & \cdots & \sqrt{n}\ov{z_n} \\
            \vdots & \ddots & \vdots\\ 
            (\sqrt{n}\ov{z_1})^{n-1}& \cdots & (\sqrt{n}\ov{z_n})^{n-1} 
        \end{pmatrix}\\
        &= 1! \cdots (n-1)!\det\Bigl( K_n(z_i, z_j)\Bigr)_{i,j=1}^n,\quad  K_n(z,w):=\sum_{l=0}^{n-1}\frac{(n z\ov{w})^l}{l!},
    \end{split}
\end{equation}
so that we conclude 
\begin{equation}
    \rho_n(\bm z) = \frac{n^n}{\pi^n n!} e^{-n\abs{\bm z}^2} \det\Bigl( K_n(z_i, z_j)\Bigr)_{i,j=1}^n,
\end{equation}
i.e.\ the eigenvalues form a \emph{determinantal process}.
Note that $K_n$ is the kernel of a positive operator of rank \(n\), in particular its off-diagonal
terms are estimated by the diagonal ones via Cauchy-Schwarz inequality:
\begin{equation}
    \abs{ K_n(z,w)}^2 \le K_n(z,z)K_n(w,w),
\end{equation}
which also follows directly from the formula for $K_n(z,w)$. In order to integrate out variables we rely on the following well-known identities: 
    \begin{equation}\label{Kn}
        \frac{n}{\pi}\int_{\C}  e^{-n\abs{z}^2} K_n(z,z)\dif^2 z =n,
    \end{equation}
    and for any fixed \(w_1,w_2\in\C\) %
    \begin{equation}\label{KnKn}
        \frac{n}{\pi}\int_{\C} e^{-n\abs{z}^2} K_n(w_1,z)K_n(z,w_2) \dif^2 z = K_n(w_1,w_2).
    \end{equation}
We recall that    both claims follow directly from the identity 
    \begin{equation}
        \begin{split}
            \frac{n}{\pi}\int_{\C} e^{-n\abs{z}^2} (\sqrt{n}z)^a (\sqrt{n}\ov{z})^b \dif^2 z = \delta_{ab} a!
        \end{split}
    \end{equation}
    for any  $a, b\in \N$
 and the definition of \(K_n\). 
As a consequence of these identities, %
 an arbitrary number of variables can be integrated out and we obtain the following
standard formula for the correlation functions:  
\begin{lemma}[\(k\)-point correlation function]  
    For
    \begin{equation}
      \rho_n^k(z_1,\ldots,z_k):=  \int_{\C^{n-k}}\rho_n(\bm z)\dif^2 z_{k+1} \cdots \dif^2 z_n
    \end{equation}
  it holds that
    \begin{equation}
        \rho_n^k(\bm z)=\frac{n^k(n-k)!}{\pi^n n!} e^{-n\abs{\bm z}^2} \det\Bigl( K_n(z_i, z_j)\Bigr)_{i,j=1}^k.
    \end{equation}
\end{lemma}

Consider a function \(g\colon\C\to[0,1]\) and evaluate  
\begin{equation}\label{prod g Fredholm}
    \begin{split}
        \E \prod_{i=1}^n (1-g(\sigma_i)) &= \int_{\C^n} \rho_n(\bm z) \prod_{i=1}^n (1-g(z_i))\dif^2\bm z \\
        &= \sum_{k=0}^n  (-1)^{k}  \binom{n}{k} \int_{\C^k}  \rho_n^k(\bm z)\prod_{i=1}^{k} g(z_i)\dif^2 \bm z \\
        &= \sum_{k=0}^n \frac{(-1)^k}{k!} \frac{n^k}{\pi^k }\int_{\C^n} e^{-n\abs{\bm z}^2}\det\Bigl( K_n(z_i, z_j)\Bigr)_{i,j=1}^n \prod_{i=1}^k g(z_i) \dif^2 \bm z\\
        & = \sum_{k=0}^n \frac{(-1)^k}{k!} \int_{\C^k} \det\Bigl( \sqrt{g(z_i)} \wt K_n(z_i, z_j)\sqrt{g(z_j)}\Bigr)_{i,j=1}^k \dif^2 \bm z\\
        &= \det(1-\sqrt{g}\wt K_n\sqrt g)
    \end{split}
\end{equation}
which we recognize as the Fredholm determinant of \(1-\sqrt{g}\wt K_n \sqrt g\) (see~\cref{def fred} below, and recall that \(\wt K_n\) has rank \(n\)), where 
\begin{equation}
    \wt K_n(z,w):= \frac{n}{\pi} e^{-n(\abs{z}^2+\abs{w}^2)/2} K_n(z,w) = \frac{n}{\pi} e^{-n(\abs{z}^2+\abs{w}^2-2 z \ov w)/2} \frac{\Gamma(n,n z\ov{w})}{\Gamma(n)}.
\end{equation} 
Here \(\Gamma(\cdot,\cdot)\) denotes the incomplete Gamma function defined as 
\begin{equation}
    \Gamma(s,z):=\int_z^\infty t^{s-1} e^{-t}\dif t,
\end{equation}
where $s\in \N$ and the integration contour goes from $z\in \C$ to real infinity. 
\begin{definition}[Fredholm determinant]\label{def fred}
    Let \((\Omega,\mu)\) denote a measure space and let \(K(z,w)\) be a kernel on \(\Omega\). Then  the \emph{Fredholm determinant} of \(1-K\) is defined as 
    \begin{equation}
        \det(1-K):=\sum_{k=0}^\infty \frac{(-1)^k}{k!} \int_{\Omega^k} \det\Bigl(K(z_i,z_j)\Bigr)_{i,j=1}^k \dif\mu(z_1)\cdots\dif\mu(z_k).
    \end{equation}
\end{definition}

\subsection{Scaling limit for \(\max \Re \sigma_i\)}
We now consider the scaling limit for the part of the complex plane in which the eigenvalue with the largest real part is located c.f.~\cref{fig circ law}. We will show that the eigenvalue with the largest real part lives on a scale
\((4\gamma n)^{-1/2}\times \ii (\gamma n)^{-1/4}\) around \(1+\sqrt{\gamma/4n}\). 

The fact that outside the unit circle the kernel \(\wt K_n\) has small Hilbert-Schmidt norm prompts the introduction of the regularised determinant~\cite[IV.(7.8)]{MR1744872}  
\begin{equation}
    \dett(1-K):=\det\Bigl((1-K)e^{K}\Bigr)
\end{equation}
which for finite-rank \(K\) allows to write \(\det(1-K) = \dett(1-K) \exp(-\Tr K)\). From~\cite[IV.(7.11)]{MR1744872} we thus conclude 
\begin{equation}\label{det tr}
    \abs{\det(1-K) - \exp(-\Tr K)}\le  \norm{K}_2 e^{(\norm{K}_2+1)^2/2-\Tr K}
\end{equation}
where 
\begin{equation}
    \Tr K = \int_\Omega K(x,x)\dif\mu(x), \quad \norm{K}_2^2 = \int_{\Omega^2} \abs{K(x,y)}^2 \dif\mu(x)\dif \mu(y).
\end{equation}
The regularized determinant as a technical tool was used in~\cite{MR3211006} 
in a very similar context for the spectral radius of real Ginibre matrices.
\begin{proposition}\label{gumbel lemma}
    Let \(\abs{t}\le \sqrt{\log n}/10\) and define the set 
    \begin{equation}\label{A def}
        A=A(t):= \set*{z\in\C\given \Re z\ge 1+\sqrt{\frac{\gamma}{4n}}+\frac{t}{\sqrt{4\gamma n}}}. %
    \end{equation}
    Then for \(g\colon\C\to[0,1]\) supported on \(\supp g\subset A(t)\), and for \(n\) large enough so that \(\gamma>0\), it holds that     
    \begin{equation}\label{eq tr K statement}
        \begin{split}
            \Tr \sqrt{g}\wt K_n \sqrt{g} &= \int_{t}^\infty \int_\R g(z) \frac{e^{-x-y^2}}{\sqrt{\pi}} \dif y\dif x + \landauO*{e^{-t}\frac{(\log\log n)^2+\abs{t}^2}{\log n}}, \\
            z&=1+\sqrt{\frac{\gamma}{4n}}+\frac{x}{\sqrt{4\gamma n}} + \frac{\ii y}{(\gamma n)^{1/4}}
        \end{split}
    \end{equation}
    and 
    \begin{equation}\label{eq hs K statement}
        \norm{\sqrt{g}\wt K_n \sqrt{g}}_2 \lesssim e^{-\sqrt{\log n}/32}.
    \end{equation}
    The unspecified constants in $\lesssim$ and $\landauO*{\cdot}$ are uniform in $n$ and  in $\abs{t}\le \sqrt{\log n}/10$.
\end{proposition}
In particular,~\cref{prod g Fredholm,det tr} combined with~\cref{gumbel lemma} for any  fixed \(t\) gives
\begin{equation}
    \begin{split}
        \Prob\Bigl(\max_i\Re\sigma_i< 1 + \sqrt{\frac{\gamma}{4n}}+\frac{t}{\sqrt{4\gamma n}} \Bigr)
        &=\Prob\Bigl(\sigma_1,\ldots,\sigma_n\in A(t)^c\Bigr) \\
        &= \det(1-\chi_{A(t)} \wt K_n \chi_{A(t)}) \xrightarrow{n\to\infty} e^{-e^{-t}}
    \end{split}
\end{equation}
with \(\chi_A\) denoting the characteristic function of the set \(A\), completing the proof of~\cref{thm gumbel} in the complex case. Moreover, for any function \(f\colon\C\to[0,\infty)\) supported in \(A(t)\) we also have that
\begin{equation}
    \begin{split}
        \E \exp\Bigl(-\sum_{i=1}^n  f(\sigma_i)\Bigr) &= \det\Bigl(1-\sqrt{1-e^{-f}}\wt K_n \sqrt{1-e^{-f}}\Bigr)\\ 
        &\xrightarrow{n\to\infty} \exp\Bigl(-\int_{t}^\infty\int_\R (1-e^{-f(z)}) \frac{e^{-x-y^2}}{\sqrt{\pi}}\dif y\dif x\Bigr)
    \end{split}
\end{equation}
with \(z\) as in~\cref{eq tr K statement}, proving the complex case of~\cref{thm poisson} after change of variables.
The error terms in~\eqref{gum} and~\eqref{pois} can easily be obtained 
from~\eqref{eq tr K statement}--\eqref{eq hs K statement}.

Hence the remaining task is to prove~\cref{gumbel lemma} which
will be an easy consequence of the following~\cref{Kernel lemma}. 
\begin{lemma}\label{Kernel lemma}
    Rescale the kernel variables as
    \begin{equation}\label{varchange}
        z = 1 + \sqrt{\frac{\gamma}{4n}} + \frac{x_1}{\sqrt{4\gamma n}} + \frac{\ii y_1}{(\gamma n)^{1/4}}, \quad w= 1 + \sqrt{\frac{\gamma}{4n}} + \frac{x_2}{\sqrt{4\gamma n}} + \frac{\ii y_2}{(\gamma n)^{1/4}}
    \end{equation}
    with ${\bm x} := (x_1, x_2)$, ${\bm y}:=(y_1, y_2)$ real vectors.
    In the regime   \(\abs{\bm x}+\abs{\bm y}^2 \le \sqrt{\log n}/2\) and  for  \(\abs{y_1-y_2}< n^{1/10} n^{-1/4}\) we have the asymptotics 
    \begin{equation}\label{off diag close}
        \frac{\abs{\wt K_n(z,w)}^2}{4(\gamma n)^{3/2}} = \frac{\gamma e^{-x_1-x_2-y_1^2-y_2^2}}{\pi(\gamma+\sqrt{n/\gamma}(y_1-y_2)^2)} \biggl(1+\landauO*{\frac{\log \log n+\abs{\bm x}^2+\abs{\bm y}^4}{\log n}}\biggr).
    \end{equation}
    On the other hand, for \(\abs{\bm x}+\abs{\bm y}^2 \le \sqrt{\log n}/2\)  and
    \(\abs{y_1-y_2}\ge C n^{-1/4}\) for some \(C\ge 1\) we have the estimate 
    \begin{equation}\label{off diag far}
        \frac{\abs{\wt K_n(z,w)}^2}{(\gamma n)^{3/2}} \lesssim
        \frac{\gamma e^{-x_1-y_1^2-x_2-y_2^2}}{\gamma+\sqrt{n/\gamma}(y_1-y_2)^2}  \biggl(1+\landauO*{\frac{\sqrt{\gamma}}{C^2}+\frac{ \abs{\bm x}^2+\abs{\bm y}^4}{\log n}}\biggr). 
    \end{equation}
    Finally, for \(x_1+y_1^2\ge0\), \(x_2+y_2^2\ge0\) we have the uniform bound 
    \begin{equation}\label{offdiag large}
        \frac{\abs{\wt K_n(z,w)}^2}{(\gamma n)^{3/2}} \lesssim \abs{z}^2\abs{w}^2 e^{-(x_1+y_1^2)/3}e^{-(x_2+y_2^2)/3}.
    \end{equation}
\end{lemma}
\begin{proof}[Proof of~\cref{gumbel lemma}]
    Set \(t_0:=4(\log\log n+\abs{t})\) and estimate the trace in~\cref{eq tr K statement} as follows
    \begin{equation}\label{tr K cplx}
        \begin{split}
            \Tr \sqrt{g}\wt K_n\sqrt g&= \int_{A(t)} g(z) \wt K_n(z,z)\dif^2 z \\
            &= \biggl(\int_{t}^{t_0}\int_{y^2< 2t_0}+\int_{t}^{t_0}\int_{y^2\ge 2t_0}+\int_{t_0}^{\infty}\int_\R\biggr) g(z) \frac{\wt K_n(z,z)}{2(\gamma n)^{3/4}}\dif y\dif x \\
            &= \int_{t}^{t_0}\int_{y^2< t_0} g(z)\frac{e^{-x-y^2}}{\sqrt{\pi}}\dif y\dif x \Bigl(1+\landauO*{\frac{(\log\log n)^2+\abs{t}^2}{\log n}}\Bigr)+\landauO*{e^{-\abs{t_0}/4}}\\ 
            &= \int_{t}^\infty \int_\R g(z)\frac{e^{-x-y^2}}{\sqrt{\pi}}\dif y\dif x +\landauO*{e^{-t}\frac{(\log\log n)^2+\abs{t}^2}{\log n}},
        \end{split}
    \end{equation}
    where we used~\cref{off diag close} for the first integral and~\cref{offdiag large} for the remaining two integrals. 
    
    For the bound on~\cref{eq tr K statement} we estimate 
    \begin{equation}
        \Tr (\sqrt{g} \wt K_n \sqrt{g})^2 \le \iint_{A(t)} \abs{\wt K_n(z,w)}^2\dif^2 z \dif^2 w,   
    \end{equation}
    and after a change of variables from $(z, w)$ to $({\bm x}, {\bm y})$  using~\eqref{varchange} we split the integral 
    into two parts.    
    First estimate the part 
    where \(\abs{\bm x}+\abs{\bm y}^2>\sqrt{\log n}/2\) and obtain 
    \begin{equation}\label{HS large cplx}
        \begin{split} 
            &\iint_{t}^\infty\iint_{\R} \frac{\abs{\wt K_n(z,w)}^2}{4(\gamma n)^{3/2}} \bm 1\bigl(\abs{\bm x}+\abs{\bm y}^2>\frac{\sqrt{\log n}}{2}\bigr) \dif \bm y \dif \bm x\\ 
            &\quad \le \int_t^\infty\int_\R \iint_\R \frac{\abs{\wt K_n(z,w)}^2}{4(\gamma n)^{3/2}}\bm 1\bigl(\abs{x_1}+y_1^2>\frac{\sqrt{\log n}}{4}\bigr) \dif \bm y \dif x_2 \dif x_1 \\
            &\quad = \int_t^\infty \int_\R \frac{\wt K_n(z,z)}{2(\gamma n)^{3/4}}\bm 1\bigl(\abs{x_1}+y_1^2>\frac{\sqrt{\log n}}{4}\bigr) \dif y_1 \dif x_1  \\
            &\quad \lesssim \int_t^\infty \int_\R e^{-(x+y^2)/4}\bm 1\bigl(\abs{x}+y^2>\frac{\sqrt{\log n}}{4}\bigr)\dif y \dif x\lesssim e^{-\sqrt{\log n}/16}
        \end{split}
    \end{equation}
    due to~\cref{KnKn} in the second and~\cref{offdiag large} in the last step. In the remaining integral we use~\cref{off diag close} whenever \(\abs{y_1-y_2}\le n^{-1/6}\) and~\cref{off diag far} otherwise to find 
    \begin{equation}\label{small}
        \begin{split}
            &\iint_{t}^\infty\iint_{\R} \frac{\abs{\wt K_n(z,w)}^2}{(\gamma n)^{3/2}} \bm 1\bigl(\abs{\bm x}+\abs{\bm y}^2\le \frac{\sqrt{\log n}}{2}\bigr) \dif \bm x\dif \bm y\\ 
            &\lesssim\iint_{t}^\infty\iint_{\R} e^{-x_1-x_2-y_1^2-y_2^2}\Bigl(\bm 1(\abs{y_1-y_2}\le n^{-1/6})+\frac{\bm 1(\abs{y_1-y_2}> n^{-1/6})}{\gamma^{-3/2}n^{1/6}}  \Bigr) \dif \bm y\dif \bm x\\
            &\lesssim e^{-2t} n^{-1/6}\gamma^{3/2},
        \end{split}
    \end{equation}
    concluding the proof. 
\end{proof}
\begin{proof}[Proof of~\cref{Kernel lemma}]
    For~\cref{offdiag large} by Cauchy-Schwarz it is sufficient to prove 
    \begin{equation}\label{diag large}
        \frac{\wt K_n(z,z)}{(\gamma n)^{3/4}} \lesssim \abs{z}^2 e^{-(x+y^2)/3}.
    \end{equation}
    For the proof of~\cref{diag large} we recall the asymptotics~\cite[Lemma 3.2]{MR3211006} of the incomplete \(\Gamma\) function
    \begin{equation}\label{gamma mu}
        \frac{\Gamma(n,nt)}{\Gamma(n)} = \frac{t \mu(t) \erfc(\sqrt{n}\mu(t))}{\sqrt{2} (t-1)}\Bigl(1+\landauO*{n^{-1/2}}\Bigr), \quad \mu(t):=\sqrt{t-\log(t)-1}, 
    \end{equation}
   which holds uniformly in \(t>1\), and note that 
    \begin{equation}
        \abs{z}^2 = 1 + \frac{\sqrt{\gamma}+(x+y^2)/\sqrt{\gamma}}{\sqrt{n}} + \frac{(\gamma+x)^2}{4\gamma n} \ge  1 + \frac{\sqrt{\gamma}+(x+y^2)/\sqrt{\gamma}}{\sqrt{n}}.
    \end{equation}
    Then, for~\cref{diag large} we use \(\erfc(x)\lesssim e^{-x^2}/x\) to estimate
    \begin{equation}
        \begin{split}
            \frac{1}{(\gamma n)^{3/4}} \wt K_n(z,z)  \lesssim \frac{n^{1/4}}{\gamma^{5/4}}\abs{z}^2 e^{-n\mu(\abs{z}^2)^2}\le \frac{n^{1/4}}{\gamma^{5/4}}\abs{z}^2 e^{-\gamma(1-\sqrt{\gamma/n})/2} e^{-(x+y^2)/3}
        \end{split}
    \end{equation}
    using the elementary bound \(t-\log t-1\ge\delta(1-\delta)(t-1)/2\) for \(t\ge 1+\delta\) and \(\delta\in[0,1)\) implying 
    \begin{equation}\label{mu ineq}
        \mu(\abs{z}^2)^2=\abs{z}^2-2\log\abs{z}-1 \ge \frac{\gamma+x+y^2}{2n}\Bigl(1-\sqrt{\frac{\gamma}{n}}\Bigr) \ge \frac{\gamma}{2n}\Bigl(1-\sqrt{\frac{\gamma}{n}}\Bigr) + \frac{x+y^2}{3n}
    \end{equation}
    due to \(\gamma/n\ll 1\) in the last step. Now~\cref{diag large} follows from 
    \begin{equation}\label{e gamma asymp}
        e^{-\gamma/2} = \exp\Bigl(-\frac{1}{4}\log \frac{n}{2\pi^4(\log n)^5}\Bigr) = \frac{2^{1/4}\pi (\log n)^{5/4}}{n^{1/4}} = \frac{2^{3/2}\pi \gamma^{5/4}}{n^{1/4}} \Bigl(1+\landauO*{\frac{\log\log n}{\log n}}\Bigr). 
    \end{equation}
    
    For~\cref{off diag far} we first note 
    \begin{equation}
        z \ov{w} = 1 +  \frac{\sqrt{\gamma}+\bigl(\frac{x_1+x_2}{2}+y_1y_2\bigr)/\sqrt{\gamma}}{\sqrt{n}} + \ii \frac{y_1-y_2}{(\gamma n)^{1/4}} +  \ii\frac{y_1(\gamma+x_2/\gamma)-y_2(\gamma+x_1/\gamma)}{(\gamma n)^{3/4}},
    \end{equation}
    and hence \(\abs{1-z\ov w}\gtrsim (\abs{y_1-y_2}(n/\gamma)^{1/4}+\sqrt{\gamma})/\sqrt{n}\). Now we use the asymptotics~\cite[Lemma 3.4]{MR3211006}
    \begin{equation}
        \frac{\Gamma(n,nz\ov{w})}{\Gamma(n)} = e^{-nz\ov{w}} \frac{e^{n}(z\ov{w})^n }{\sqrt{2\pi n}(1-z\ov{w})} \biggl(1+\landauO*{\frac{1}{n\abs{1-z\ov{w}}^2}}\biggr) 
    \end{equation}
    to estimate 
    \begin{equation}
        \begin{split}
            \frac{\abs[\big]{\wt K_n(z,w)}}{(n\gamma)^{3/4}} & \lesssim \frac{n^{1/4}e^{n(1-\abs{z}^2/2-\abs{w}^2/2+\log\abs{z\ov w})}}{\gamma^{3/4}(\abs{y_1-y_2}(n/\gamma)^{1/4}+\sqrt{\gamma})}\biggl(1+\landauO*{\frac{\sqrt{\gamma}}{C^2}}\biggr) .
        \end{split}
    \end{equation}
    In the exponent we use 
    \begin{equation}
        \begin{split}
            1-\abs{z}^2/2-\abs{w}^2/2+\log\abs{z} + \log \abs{w} &= -\frac{(\abs{z}^2-1)^2}{4}-\frac{(\abs{w}^2-1)^2}{4} + \landauO*{(\gamma/n)^{3/2}} \\
            &= - \frac{\gamma +  x_1 +  y_2^2 + x_2 + y_2^2}{2n} + \landauO*{\frac{1+\abs{\bm x}^2+\abs{\bm y}^4}{n\gamma}}
        \end{split}
    \end{equation}
    to conclude 
    \begin{equation}
        \frac{\abs[\big]{\wt K_n(z,w)}}{(n\gamma)^{3/4}} \lesssim \frac{\sqrt\gamma e^{-(x_1+y_1^2)/2}e^{-(x_2+y_2^2)/2}}{\abs{y_1-y_2}(n/\gamma)^{1/4}+\sqrt{\gamma}}  \biggl(1+\landauO*{\frac{\sqrt{\gamma}}{C^2}+\frac{\log\log n+x^2+y^4}{\log n}}\biggr).
    \end{equation}
    
    It remains to consider~\cref{off diag close} where we use~\cite[Lemma 3.3]{MR3211006} in the form 
        \begin{equation}\label{gamma mu cplx}
        \frac{\Gamma(n,nz\ov{w})}{\Gamma(n)} = \frac{z\ov w \mu(z\ov w) \erfc(\sqrt{n}\mu(z\ov w))}{\sqrt{2} (z\ov w-1)}\Bigl(1+\landauO*{\frac{1}{n\abs{1-z\ov w}} }\Bigr), \quad \mu(z):=\sqrt{z-\log(z)-1}.
    \end{equation}
   We use the Taylor expansion \(\mu(1+z)=z/\sqrt{2}+\landauO{\abs{z}^2}\)
    (for small enough \(\abs{z}\)) and the asymptotics~\cite[Eq.~(7.12.1)]{DLMF} of the error function \(\erfc(z)=e^{-z^2}/(\sqrt{\pi}z)(1+\landauO*{\abs{z}^{-2}})\) for \(\abs{\arg z}<3\pi/4\) to obtain
    \begin{equation}
        \begin{split}
            \frac{\Gamma(n,nz\ov{w})}{\Gamma(n)} &=  \frac{e^{-n(z\ov w-1)^2/2 }}{\sqrt{2\pi}\sqrt{n}(z\ov w-1)} \biggl(1+\landauO*{\abs{z\ov w-1}+n\abs{z\ov w-1}^3+\frac{1}{n\abs{z\ov w-1}^2}}\biggr) 
        \end{split}
    \end{equation}
    and thereby
    \begin{equation}
        \begin{split}
            \frac{\abs[\big]{\wt K_n(z,w)}^2}{4(n\gamma)^{3/2}}  &= \frac{n^{1/2}}{\gamma^{3/2}(2\pi)^{3}} \frac{e^{-\gamma-x_1-x_2-y_1^2-y_2^2}  }{\gamma+\sqrt{n/\gamma}(y_1-y_2)^2}\biggl(1+\landauO*{\frac{1+\abs{\bm x}^2+\abs{\bm y}^4}{\gamma}}\biggr) \\
            & = \frac{\gamma e^{-x_1-x_2-y_1^2-y_2^2}}{\pi(\gamma+\sqrt{n/\gamma}(y_1-y_2)^2)}  \biggl(1+\landauO*{\frac{\log \log n+\abs{\bm x}^2+\abs{\bm y}^4}{\log n}}\biggr).
        \end{split}
    \end{equation}
 Here we also used the upper bound on \(\abs{y_1-y_2}\le n^{1/10}n^{-1/4}\)
  in order to estimate \(\sqrt{\gamma/n}\lesssim\abs{1-z\ov w}\lesssim n^{-1/2}(\sqrt{\gamma}+n^{1/10}/\gamma^{1/4})\).
\end{proof} 

\newcommand{\cor}{\color{red}}

\section{Real Ginibre}
We now consider the real case. The analogue of~\cref{prod g Fredholm} for test functions \(g\colon\C\to[0,1]\) invariant under complex conjugation, \(g(\ov z)=g(z)\), and vanishing on the real line, \(g(x)=0\), \(x\in\R\), is given by~\cite{MR3211006} 
\begin{equation}\label{prod g Fredholm real}
    \begin{split}
        \E \prod_{i=1}^n (1-g(\sigma_i)) & = %
        \bigl[\det(1-\sqrt{g}K_n^{\C,\C}\sqrt{g})\bigr]^{1/2},
    \end{split}
\end{equation}
where 
\begin{equation}\label{realkernel}
    K_n^{\C,\C}(z,w):=\begin{pmatrix}
        S_n(z,w)&-\ii S_n(z,\ov w)\\-\ii S_n(\ov z,w)& S_n(w,z)
    \end{pmatrix}
\end{equation}
with
\begin{equation}
    \begin{split}
        S_n(z,w):={}& \frac{\ii  n e^{-n(z-\ov{w})^2/2}}{\sqrt{2\pi}} \sqrt{n}(\ov{w}-z)\sqrt{\erfc(\sqrt{2n}\abs{\Im z})\erfc(\sqrt{2n}\abs{\Im w})} e^{-nz\ov w} K_{n}(z,w)\\
        ={}&\Phi_n(z,w)  \wt K_n(z,w) \\
        \Phi_n(z,w):={}&e^{n(\abs{z}^2+\abs{w}^2-2z\ov w)/2}\frac{\ii \sqrt{\pi} e^{-n(z-\ov{w})^2/2}}{\sqrt{2}} \sqrt{n}(\ov{w}-z)\sqrt{\erfc(\sqrt{2n}\abs{\Im z})\erfc(\sqrt{2n}\abs{\Im w})}.
    \end{split}
\end{equation}
The analogue to~\cref{gumbel lemma} is the following result. 
\begin{proposition}\label{gumbel lemma real}
    Let \(\abs{t}\le \sqrt{\log n}/10\), let \(A(t)\) be as in~\cref{A def} and recall $\gamma=\gamma_n$ from~\cref{Cmax}. 
    Consider any function \(g\colon\C\to[0,1]\) supported on  \(\supp g\subset A(t)\) that is 
     symmetric  in the sense \(g(z)=g(\ov z)\), %
     and let \(n\) be large enough such
     that \(\gamma>0\). Then we have
         \begin{equation}\label{eq tr K statement real}
        \begin{split}
            \Tr \sqrt{g}K_n^{\C,\C} \sqrt{g} &= 2\int_{t}^\infty \int_0^\infty g(z) \frac{e^{-x-y^2}}{\sqrt{\pi}} \dif y\dif x + \landauO*{e^{-t}\frac{(\log\log n)^2+\abs{t}^2}{\log n}}, \\
            z&=1+\sqrt{\frac{\gamma}{4n}}+\frac{x}{\sqrt{4\gamma n}} + \frac{\ii y}{(\gamma n)^{1/4}}
        \end{split}
    \end{equation}
    and 
    \begin{equation}\label{eq hs K statement real}
        \norm{\sqrt{g} K_n^{\C,\C} \sqrt{g}}_2 \lesssim e^{-\sqrt{\log n}/32}.
    \end{equation}
    The unspecified constants in $\lesssim$ and $\landauO*{\cdot}$ are uniform in $n$ and  in $\abs{t}\le \sqrt{\log n}/10$.
\end{proposition}
\begin{proof}
    We estimate 
    \begin{equation}\label{eq Phi diag est}
        \begin{split}
            \Phi_n(z,z) = \sqrt{\pi} e^{2n (\Im z)^2} \sqrt{2n}\Im z \erfc(\sqrt{2n}\abs{\Im z})  =1 + \landauO*{\min\set*{1,\frac{1}{n(\Im z)^2}}},
        \end{split}
    \end{equation}
    where we used the asymptotic \(\erfc(x)=e^{-x^2}/(\sqrt{\pi}x)(1+\landauO*{x^{-2}})\) and the bound \(\erfc(x)\le e^{-x^2}/(\sqrt{\pi}x)\). Thus the tracial computation essentially reduces to the complex case~\cref{tr K cplx} and we obtain 
    \begin{equation}
        \begin{split}
            \Tr \sqrt{g} K_n^{\C,\C}(z,z) \sqrt{g} &= 2 \int_{A(t)_+} g(z)  S_n(z,z) \dif^2 z \\
            &= 2 \int_{A(t)_+} g(z) \wt K_n(z,z)\bm1(\Im z>n^{-5/12})\dif^2 z \Bigl(1+\landauO*{n^{-5/12}}\Bigr) \\
            &\qquad + \landauO*{ \int_{A(t)_+} \wt K_n(z,z)\bm1(\Im z\le n^{-5/12})\dif^2 z } \\
            &= 2\int_{t}^\infty \int_{0}^\infty \frac{e^{-x-y^2}}{\sqrt{\pi}} g(z)\dif y \dif x + \landauO*{e^{-t}\frac{(\log\log n)^2+\abs{t}^2}{\log n}},
        \end{split}
    \end{equation} 
    where \(A(t)_+:=A(t)\cap\HC\), we parametrized $z$ with $x, y$ as in~\cref{eq tr  K statement}, and we used~\cref{off diag close,offdiag large}.
    
    For the Hilbert-Schmidt norm we estimate, analogously to~\cref{HS large cplx}
    \begin{equation}
        \begin{split}
            \norm{\sqrt{g}K_n^{\C,\C}\sqrt{g}}_2 &= \iint g(z)g(w) \Tr K_n^{\C,\C}(z,w)K_n^{\C,\C}(w,z)\dif^2 z\dif^2 w \\
            &\le \iint_{\Re\ge t} \Tr  K_n^{\C,\C}(z,w)K_n^{\C,\C}(w,z) \bm1(\abs{\vx}+\abs{\vy}^2\le \frac{\sqrt{\log n}}{2}) \dif^2 z\dif^2 w \\
            &\quad + \int \Tr  K_n^{\C,\C}(z,z)\bm1(\abs{x}+y^2>\frac{\sqrt{\log n}}{4}) \dif^2 z\\
            &\lesssim \iint_{\Re\ge t} \abs{S_n(z,w)}^2 \bm1(\abs{\vx}+\abs{\vy}^2\le \frac{\sqrt{\log n}}{2}) \dif^2 z\dif^2 w  + e^{-\sqrt{\log n}/16},
        \end{split}
    \end{equation} 
    where we used that the integrals of \(\abs{S_n(z,w)}^2\) and \(\abs{S_n(z,\ov w)}^2\) 
    are equal
    by symmetry of the integration region, and \(\Re\ge t\) indicates the integration region \(\set{\Re z\ge t}\cap\set{\Re w\ge t}\). Now we use~\cref{off diag close,off diag far} together with  the elementary bound
       \begin{equation}
        \abs{\Phi_n(z,w)}^2\lesssim \frac{n\abs{z-\ov w}^2}{(1\vee \sqrt{n}\Im z)(1\vee \sqrt{n}\Im w)} \lesssim \frac{(x_1-x_2)^2/\sqrt{\gamma}+\sqrt{n}(y_1+y_2)^2}{(\gamma^{1/4}\vee n^{1/4}y_1)(\gamma^{1/4}\vee n^{1/4}y_2)}
    \end{equation}
    to estimate 
    \begin{equation}
        \begin{split}
            \frac{\abs{S_n(z,w)}^2}{(\gamma n)^{3/2}}\lesssim \frac{\gamma e^{-x_1-x_2-y_1^2-y_2^2}}{\gamma+\sqrt{n/\gamma}(y_1-y_2)^2}  \frac{(x_1-x_2)^2/\sqrt{\gamma}+\sqrt{n}(y_1+y_2)^2}{(\gamma^{1/4}\vee n^{1/4}y_1)(\gamma^{1/4}\vee n^{1/4}y_2)}
        \end{split}
    \end{equation}
   and conclude, similarly to~\cref{small}, that  
    \begin{equation}
        \begin{split}
            \iint_{\Re\ge t} \abs{S_n(z,w)}^2 \bm1(\abs{\vx}+\abs{\vy}^2\le \frac{\sqrt{\log n}}{2}) \dif^2 z\dif^2 w \lesssim e^{-2t}n^{-1/6} \gamma^{3/2}.
        \end{split}
    \end{equation} 
\end{proof}

As a consequence of~\cref{det tr,prod g Fredholm real,gumbel lemma real} and we obtain that for any fixed \(t\) it holds that 
\begin{equation}\label{Gumbel real eq}
    \begin{split}
        \Prob\Bigl(\max_{i\colon \sigma_i\not\in\R}\Re\sigma_i< 1 + \sqrt{\frac{\gamma}{4n}}+\frac{t}{\sqrt{4\gamma n}} \Bigr)
        &=\Prob\Bigl(\sigma_1,\ldots,\sigma_n\in \R\cup [\C\setminus (A(t)\cup\R)]\Bigr) \\
        &= \bigl[ \det(1-\chi_{A(t)} K_n^{\C,\C} \chi_{A(t)})\bigr]^{1/2} \xrightarrow{n\to\infty} e^{-e^{-t}/2}
    \end{split}
\end{equation}
with \(\chi_A\) denoting the characteristic function of the set \(A\), using that \(\int_0^\infty e^{-y^2}\dif y=\sqrt{\pi}/2\). Moreover,  for any symmetric function \(f\colon\C\to[0,\infty)\) supported in \(A(t)\) we also have that 
\begin{equation}\label{Poisson real eq}
    \begin{split}
        \E \exp\Bigl(-\sum_{i\colon \sigma_i\not\in\R}  f(\sigma_i)\Bigr) &= \det\Bigl(1-\sqrt{1-e^{-f}} K_n^{\C,\C} \sqrt{1-e^{-f}}\Bigr)^{1/2}\\ 
        &\xrightarrow{n\to\infty} \exp\Bigl(-\int_{t}^\infty\int_0^\infty (1-e^{-f(z)}) \frac{e^{-x-y^2}}{\sqrt{\pi}}\dif y\dif x\Bigr).
    \end{split}
\end{equation}

In order to complete the proof of~\cref{thm gumbel,thm poisson} it remains to estimate the real eigenvalues. 
However, the real eigenvalues affect neither of these results since the largest real eigenvalue lives on a smaller scale, $1+ \landauO{1/\sqrt{n}}$, than the largest real part of complex eigenvalues, $1+ \landauO{\sqrt{\log n/n}}$.
 Indeed, the main result of~\cite{MR3678474} is that for large \(t\)
\begin{equation}\label{Gumbel real real eq}
    \lim_{n\to\infty}\Prob\Bigl(\max_{i\colon\sigma_i\in\R}\sigma_i\le 1+\frac{t}{\sqrt{n}}\Bigr) = 1-\frac{1}{4}\erfc(t) + \landauO*{e^{-2t^2}}.
\end{equation}
Together with~\cref{Gumbel real eq,Poisson real eq} this concludes the proof of~\cref{thm gumbel,thm poisson} also in the real case.

\printbibliography%
\end{document}